\documentclass[a4paper,12pt]{article}
\usepackage{amsmath,amssymb,amsthm,graphicx,color}
\numberwithin{equation}{section}
\newtheorem{thm}[equation]{Theorem}

\newtheorem{prop}[equation]{Proposition}

\newtheorem{lem}[equation]{Lemma}

\newtheorem{remk}[equation]{Remark}
\newcounter{mycount}
\newenvironment{romlist}{\begin{list}{\rm(\roman{mycount})}
   {\usecounter{mycount}\labelwidth=1cm\itemsep 0pt}}{\end{list}}

\DeclareMathSymbol{\leqslant}{\mathalpha}{AMSa}{"36}
\DeclareMathSymbol{\geqslant}{\mathalpha}{AMSa}{"3E}
\renewcommand{\le}{\;\leqslant\;}
\renewcommand{\ge}{\;\geqslant\;}

\def\O{{\mathrm O}}

\def\qq{\qquad}
\def\d{\delta}

\def\ZZ{{\mathbb Z}}

\def\RR{{\mathbb R}}

\def\EE{{\mathbb E}}
\def\PP{{\mathbb P}}

\def\Om{\Omega}
\def\om{\omega}
\def\De{\Delta}

\def\oo{\infty}

\def\sC{\mathcal C}

\def\al{\alpha}

\def\be{\begin{equation}}
\def\ee{\end{equation}}
\def\eps{\varepsilon}

\def\no#1{\|#1\|_1}
\def\nt#1{\|#1\|_2}

\def\sm{\setminus}

\def\var{\mathrm{var}}

\def\ivf{I_\v(f)}

\def\vectornotation#1{{\mathsf #1}}
\def\0{\vectornotation 0}
\def\e{\vectornotation e}
\def\p{\vectornotation p}
\def\q{\vectornotation q}
\def\u{\vectornotation u}
\def\v{\vectornotation v}

\def\x{\vectornotation x}
\def\y{\vectornotation y}
\def\z{\vectornotation z}

\def\Z{\vectornotation Z}
\def\a{\vectornotation a}

\def\b{\vectornotation b}
\def\ml#1{\sloppy\mbox{#1}\fussy}

\def\qtext#1{\quad\text{#1}\quad}

\def\lra{\leftrightarrow}

\def\bcdot{\,\cdot\,}
\def\pfo{\tfrac{\partial f}{\partial \om(\v)}}
\def\zzd{\ZZ^d}
\def\zzdm{\ZZ^{d-1}}

\begin{document}
\title{Sublinear variance for directed last-passage percolation}
\author{B.\ T.\ Graham\\
   {\small DMA--\'Ecole Normale Sup\'erieure, 45 rue d'Ulm}\\
   {\small 75230 Paris Cedex 5, France}\\
   {\small \tt graham@dma.ens.fr}
}
\maketitle
\begin{abstract}
\noindent
A range of first-passage percolation type models are believed to demonstrate the related properties of sublinear variance and superdiffusivity.  We show that directed last-passage percolation with Gaussian vertex weights has a sublinear variance property. We also consider other vertex weight distributions.

Corresponding results are obtained for the ground state of the `directed polymers in a random environment' model.\\

\\

\end{abstract}

\section{First- and last-passage percolation}
Let $G=(V,E)$ be a graph, and let $\om=(\om(e):e\in E)$ be a collection of independent, identically distributed edge weights. The weight of a path from $x$ to $y$ is the sum of the edge weights along the path. The first-passage time $S(x,y)$ is defined to be the weight of the lightest path from $x$ to $y$.

Consider first-passage percolation on the $d$-dimensional hypercubic lattice $\ZZ^d$ with nearest-neighbor edges. When $d=1$, the variance of $S(0,n)$ is proportional to $|n|$. In contrast, for $d\ge 2$, the variance of $S(\0,\x)$ is sublinear as a function of $|\x|$ for a wide range of edge weight distributions \cite{BR,BKS}. 

Directed last-passage percolation is a variant of first-passage percolation that is defined on directed lattices; for an introduction see \cite{MartinSurvey}.  Let \ml{$(\e_i:i=1,\dots,d)$} denote the standard basis for $\RR^d$.  
The set
\be\label{E def}
\vec{E}=\vec{E}(\zzd):=\{\x\to\x+\e_i:\x\in\zzd\text{ and }i=1,\dots,d\}
\ee
of directed nearest-neighbor edges turns $\zzd$ into a directed graph.
The coordinate-wise partial order on $\zzd$ is defined by $\x\le\y$ iff $x_i\le y_i$ for $i=1,\dots,d$. A directed path exists from $\x$ to $\y$ iff $\x\le\y$. Let $\om=(\om(\v):\v\in\zzd)$ be independent, identically distributed vertex weights. For $\x\le\y\in\zzd$, the last-passage time is defined
\[
T(\x,\y)=\sup_\gamma \sum_{\v\in \gamma} \om(\v);
\]
the supremum is over directed paths $\gamma$ from $\x$ to $\y$. By convention, $\gamma$ is identified with the set of vertices along the path, but with the starting point $\x$ excluded. On $\RR^d$, we will use $|\bcdot|$ to denote the $L_1$ norm: all of the directed paths from $\x$ to $\y$ have length $|\x-\y|$. 

Assume now that $\EE(|\om(\v)|)<\oo$. By \cite[Proposition 2.1]{MartinSurvey}, we can define the asymptotic last-passage time in direction $\x\in\RR_+^d$,
\begin{align}\label{lim_def}
g(\x)=\lim_{N\to\oo} T(\0,\lfloor N\x\rfloor)/N \in(-\oo,\oo].
\end{align}
The resulting function $g$ is homogeneous of degree 1 and concave. 
If $\int_0^\oo \PP(\om(\v)>t)^{1/d}\,{\mathrm d}t$ is finite, then $g$ is finite and continuous \cite[Theorems 4.1 and 4.4]{MartinSurvey}. 

Much of the interest in first- and last-passage percolation is due to the related phenomena of sublinear variance and superdiffusivity. The most complete results are available in the case of directed last-passage percolation on $\ZZ^2$ with geometric vertex weights \cite{BDLPX,JohanssonShapeFluctuations}. 
That particular model is well understood due to its relationship with random matrices---the function $g$ is even known exactly. The variance of $T(\0,N(\e_1+\e_2))$ has order $N^{2/3}$ as $N\to\oo$; this surprising behavior is known as sublinear variance. Let $\gamma$ denote a maximal path from $\0$ to $N(\e_1+\e_2)$. The typical deviations of $\gamma$ from the line $x=y$ are of order $N^{2/3}$; this is known as superdiffusivity.

Other models display remarkably similar behavior, and so are said to belong to the same universality class.  Let $P$ denote a Poisson process with rate $1$ on the square $[0,N]^2$. Consider sequences of points in $P$ that are increasing with respect to the coordinate-wise partial order, and let $\gamma$ denote such a sequence with maximal length. The variance of the length of $\gamma$ has order $N^{2/3}$, and $\gamma$ is superdiffusive \cite{JohanssonPermutations}.

It is believed that first- and last-passage percolation belong to this universality class for a wide range of weight distributions. However, proving that in general $\var[S(\0,\x)]$ has order $|\x|^{2/3}$ seems to be rather difficult. Using Talagrand's work on concentration \cite{Talagrand_Russo}, Benjamini, Kalai and Schramm showed that for first-passage percolation with Bernoulli-type edge weights, $\var[S(\0,\x)]=\O(|\x|/\log|\x|)$ \cite{BKS}. Using a more sophisticated concentration result, Bena\"im and Rossignol extended this result to a range of `gamma-like' edge weight distributions \cite{BR}.

In this paper we adapt the arguments from \cite{BR,BKS} to show that directed last-passage percolation with Gaussian vertex weights has a sublinear variance property.  
\begin{thm}\label{variance_thm}
Let $d\ge 2$ and let $\om=(\om(\x):\x\in\zzd)$ be a collection of independent, standard normal random variables. Let $\u=\e_1+\dots+\e_d$. For directed last-passage percolation on $\zzd$,
\[
\var[T(\0,N\u)]=\O(N/\log N).
\]
\end{thm}
The proof of Theorem \ref{variance_thm} can be adapted to other vertex weight distributions including the uniform $[0,1]$ and gamma distributions. We discuss this in Section \ref{other}. 

The main ingredient in our proof is the theory of concentration. In the present context, the concentration result of Bena\"im and Rossignol can be written in a particularly simple form. We do this by extending the concept of influence from the discrete to the continuous setting---see Section \ref{formal_defs}. 

It is believed that for many vertex weight distributions, the asymptotic traversal time is strictly concave. However, in general it is an open problem to determine which vertex weight distributions produce strictly concave $g$. In Section \ref{sec:not_flat}, we show that $g(\x)/|\x|$ is increasing with the `dimensionality' of $\x$. This excludes one of the ways $g$ can fail to be strictly concave, and it suffices for the purpose of proving Theorem \ref{variance_thm}---see Section \ref{variance_section}.

Directed last-passage percolation is related to another model defined on $\ZZ^d$: directed polymers in a random environment \cite{CometsShigaYoshida,Piza}. A polymer is a directed path, starting from the origin and with fixed length. The ground state of the directed polymer model corresponds very closely to directed last passage percolation. The random environment is the collection of vertex weights. The weight of a polymer is the sum of the weights along its length, and the objects of interest are polymers of maximal weight.

There are actually two versions of the model. Piza \cite{Piza} consider the random polymer model on $(\zzd,\vec{E})$. Most other studies of the directed polymer model have considered a different directed graph. Think of $\zzd$ as space-time, with $d-1$ nearest-neighbor space dimensions and one time dimension. 

In two dimensions, the two graphs are actually equivalent, up to rotation by 45$^\circ$.
In \cite{0810.4221}, Chatterjee shows that the ground state of the two-dimensional directed polymer model with Gaussian vertex weights has the sublinear variance property, and asks if the same is true in higher dimensions. 
In both cases, the answer is yes. See Section \ref{sec dp}.

\section{Notation}
From now on, the coordinate axes of $\RR^d$ will be labeled $x_1,x_2,\dots,x_d$. 

Let $\PP=\bigotimes_{i\in I} \PP_i$ denote a product measure that supports the vertex weights, and possibly some auxiliary random variables. Let $\EE$ denote the corresponding expectation operator, and let $\Om$ denote the state space. 
For $i\in I$, define operators $\De_i$,
\[
\forall f\in L^2(\PP),\quad \De_i f = f - \int f\ {\mathrm d}\PP_i.
\]
For a set $S$ and functions $f,g:S\to\RR$, we will write
\begin{align*}
f(x)&=\O(g(x)) \text{ if } \exists\, c>0 \text{ such that } f(x) \le cg(x) \ \,\text{ when } g(x)\ge c\,,\text{ and}\\
f(x)&=\Om(g(x)) \text{ if } \exists\, c>0 \text{ such that } f(x) \ge g(x)/c \text{ when } g(x)\ge c\,.
\end{align*}
For example, $a+1/(1+b^2) = \O(a)$ for $(a,b)\in\RR^2$.

\section{Influence and concentration}\label{formal_defs}
Martin \cite[Theorem 5.1]{MartinLimitingShapeDP} applied concentration techniques to directed last-passage percolation to prove a `shape' theorem:
\begin{prop}\label{shape proposition}
If 
\be\label{shape proposition condition}
\int_0^\oo \PP(|\om(\v)|>t)^{1/d}\, {\mathrm d}t <\oo,
\ee
then $\EE[T(\0,\x)-g(\x)]/|\x|\to 0$ as $|\x|\to\oo$, and, with probability $1$
\begin{align*}
\sup_{\x\, :\, |\x|\ge N} |T(0,\x)-g(\x)|/|\x| \to 0 \qtext{as} N\to\oo.
\end{align*}
\end{prop}
The theory of concentration is also central to the arguments of \cite{BR,BKS}. In \cite{BKS}, Talagrand's concentration inequality for Bernoulli random variables 
is used to demonstrate the sublinear variance property, and also to show that first-passage times are concentrated about their average values. 
The arguments in \cite{BKS} are remarkably concise. A simplifying feature in the study of concentration for Bernoulli random variables is the concept of influence. When $\PP$ is a Bernoulli product measure, the influence of coordinate $\om(i)$ on function $f:\Om\to\RR$ is defined to be 
\be\label{inf_discreet}
I_i(f)=\PP(\De_i f\not=0)=\PP(f(\om) \text{ depends on } \om(i)).
\ee 
In \cite{BR} a more sophisticated concentration inequality \cite[Proposition 2.1]{BR} is used to extend the results of \cite{BKS} to a range of different vertex weight distributions, and also to improve the first-passage-times concentration result.

Concentration plays an even more crucial role in this paper, allowing us to deal with 
the additional complications arising in a directed space. 
In order to proof Theorem \ref{variance_thm}, we will strengthen Proposition \ref{shape proposition} in the Gaussian case, showing that the last-passage times are tightly concentrated about their mean values.
\begin{lem}\label{conc_d}
Consider Gaussian vertex weights.
For $\x\in\zzd$ and $t>0$, 
\be
\label{mean_conc}
\PP\Big(
\big|T(\0,\x)-\EE T(\0,\x)\big|
\ge t\sqrt{|\x|}\Big)
=\exp(-\Om(t^2)).  
\ee
\end{lem}
The proof of Lemma \ref{conc_d} is given at the end of this section.
We will first extend the concept of influence from Bernoulli to Gaussian random variables. 
This allows the proofs of Theorem \ref{variance_thm} and Lemma \ref{conc_d} to resemble closely the corresponding results in \cite{BKS}. 

Let $\PP$ denote a product measure that support Gaussian vertex weights.
Let $H_1^2$ denote the log Sobolev space corresponding to $\PP$ \cite{BR}. 
Define the influence of $\om(\v)$ on a function $f\in H_1^2$ to be
\[
\ivf=\PP\big(\pfo\not=0\big), \qq\v\in\zzd.
\]
With $\x\in\zzd_+$, consider $f(\om)=T(\0,\x)$. 
The influences for $f$ have a geometric interpretation. The vertex weight distribution is continuous, so the path $\gamma$ from $\0$ to $\x$ corresponding with $f$ is unique. If a vertex $\v$ lies in $\gamma$, then $\pfo=1$, otherwise $\pfo=0$. Therefore $\ivf$ is the probability that $\v\in\gamma$.
\begin{lem}\label{conc_lem}
Let $(\om_s:s\in S)$ denote auxiliary Bernoulli($1/2$) random variables.
Let $\Z:\Om\to\zzd$ denote a random variable that only depends on the auxiliary variables, so that $\Z$ is independent of the vertex weights. 
With $-\oo\le A<B\le \oo$, let
\[
f(\om)=T(\Z,\x+\Z)\wedge B \vee A, \qq I_S(f)=\sum_{s\in S} \nt{\De_s f}^2.
\]
and let $u=\PP(A<T(\0,\x)<B)$. There is a constant $c_G$ such that 
\begin{align}
\label{normal_conc4}
\var[f]\log\frac{\var[f]}{c_G^2 u |\x| \max_\v \ivf + I_S(f)} \le 2 u |\x| +2I_S(f).
\end{align}
\end{lem}
\noindent When $(A,B)=(-\oo,\oo)$, $f$ is the last-passage time from $\Z$ to $\x+\Z$. By translation invariance, $T(\Z,\x+\Z)$ has the same distribution as $T(\0,\x)$; the reason for allowing the starting point to be randomized is to reduce the maximum vertex influence $\max_{\v} \ivf$.

To see how inequality \eqref{normal_conc4} places an upper bound on $\var[f]$, note that if $b>a>0$ and $x>0$, 
\be\label{lambert}
x\log \frac{x}{a}\le b \implies x\le \frac{2b}{\log (b/a)}.
\ee
\begin{proof}[Proof of Lemma \ref{conc_lem}]
Corollary 2.2 \cite{BR} states that for $f\in H_1^2$,
\begin{align*}
\var[f]\log
\frac{\var[f]}
{
\sum_{\v\in \zzd} \no{\De_\v f}^2+
\sum_{s\in S} \no{\De_s f}^2
}
\le
2\sum_{\v\in \zzd} \nt{\pfo}^2+
2\sum_{s\in S} \nt{\De_s f}^2.
\end{align*}
The maximal path $\gamma$ corresponding to $T(\Z,\x+\Z)$ is almost surely unique. For a point $\v\in\zzd$, if $\v\in\gamma$ and $A<T(\Z,\x+\Z)<B$ then $\pfo=1$, otherwise $\pfo=0$; hence
\[
\sum_\v \nt{\pfo}^2=\sum_\v \ivf= u|\x| \qtext{and} \sum_\v \ivf^2 \le u|\x|\max_\v \ivf.
\]
To derive inequality \eqref{normal_conc4} from the above, we must show that for some constant $c_G$,
\be\label{cg1}
\forall \v\in\zzd,\qq \no{\De_\v f}\le c_G \ivf.
\ee
For $\om\in\Om$, let $\om^{-\v}=(\om(i):i\in I\sm\{\v\})$. Conditional on $\om^{-\v}$, there is an interval $(a,b)$ such that $\pfo=1$ if $a<\om(\v)<b$, and $\pfo=0$ if $\om(\v)<a$ or $\om(\v)>b$. Suppose that
\be\label{cg2}
\om^{-\v}\text{--a.s},\qq \EE_\v|\De_\v(\om(\v)\wedge b\vee a)|\le c_G\, \PP_\v(a<\om(\v)<b);
\ee
inequality \eqref{cg1} follows by integrating over $\om^{-\v}$. To check inequality \eqref{cg2}, let $X=\om(\0)$, so that $X$ represents a typical vertex weight.
We must check that $c_G<\oo$, where
\[
c_G:= \sup_{-\oo\le a<b\le \oo}\frac{\EE|X\wedge b \vee a - \EE(X\wedge b \vee a)|}{\PP(a<X<b)}.
\]
By symmetry, we can assume $-b\le a < b$. It is enough to consider the cases
\begin{align*}
\text{(i)}&\ 1<a<b< a+a^{-2},\qq& 
\text{(ii)}&\  1<a<a+a^{-2}\le b, \\
\text{(iii)}&\ |a|\le 1,\ b-a<1,\qq& 
\text{(iv)}&\ a\le 1,\ b-a\ge 1.
\end{align*}
Let $\phi(x)=\exp(-x^2/2)/\sqrt{2\pi}$. In cases (i) and (ii), by the triangle inequality and Jensen inequality,
\begin{align*}
\EE|X\wedge b \vee a -\EE(X\wedge b \vee a)|&\le \EE|X\wedge b \vee a - a-\EE(X\wedge b \vee a-a)|\\&\le 2\EE|X\wedge b \vee a-a|.
\end{align*}
In case (i), $\EE|X\wedge b \vee a-a|\le (b-a)\PP(X>a)$ and $\PP(X>a)\le \phi(a)$ so
\[
\frac{\EE|X\wedge b \vee a - \EE(X\wedge b \vee a)|}{\PP(a<X<b)}\le \frac{2(b-a)\phi(a)}{\PP(a<X<b)}\le
 \frac{2\phi(a)}{\phi(a+a^{-2})}\le
 2e^{3/2}.
\]
In case (ii), $\EE|X\wedge b \vee a-a|\le \phi(a)/a^2$ so
\begin{align*}
\frac{\EE|X\wedge b \vee a-\EE(X\wedge b \vee a)|}{\PP(a<X<b)}\le \frac{2\phi(a)/a^2}{\phi(a+a^{-2})/a^2}\le 2e^{3/2}.
\end{align*}
Cases (iii) and (iv) are simpler so we omit the details. 
\end{proof}
\begin{proof}[Proof of Lemma \ref{conc_d}]
The result is obtained by applying Lemma \ref{conc_lem} iteratively to the tails of the distribution of $T(\0,\x)$. 
Let $h$ denote the inverse cumulative distribution function for $T(\0,\x)$,
\begin{align*}
\PP[T(\0,\x)<h(u)]=u,\qquad u\in (0,1).
\end{align*}
Apply Lemma \ref{conc_lem} to $f(\om)=T(\0,\x)\wedge B \vee A$ with $A=h(u)$ and $B=h(2u)$. For any $\v$, $\pfo=1$ implies $A<T(\x,\y)<B$; hence $\max_\v\ivf\le u$. By \eqref{lambert},
\[
\var[f] \le 4u|\x|/\log(2c_G^{-2}u^{-1}).
\]
By Chebyshev's inequality, $h(2u)-h(u)=\sqrt{|\x|}/\Om(\sqrt{\log 1/u})$. By a telescopic sum, $h(2^{-1})-h(2^{-n})=\sqrt{|\x|}\,\O(\sqrt{n})$. In the same way, $h(1-2^{-n}) - h(1-2^{-1})=\sqrt{|\x|}\,\O(\sqrt{n})$. Hence 
\[
\PP\Big( \big| T(0,\x)-h(1/2)\big|\ge t\sqrt{|\x|} \Big)=\exp(-\Om(t^2)).
\]
This implies that $|h(1/2)-\EE T(0,\x)|=\O(\sqrt{|\x|})$ and \eqref{mean_conc} follows.
\end{proof}

\section{Concavity of $g$\label{sec:not_flat}}
Assume that the vertex weight distribution has a finite mean, 
so that the asymptotic last-passage time function $g$, defined in \eqref{lim_def}, is homogeneous of degree 1 and concave.
Last-passage percolation is symmetric with respect to permutations of $(\e_i:i=1,\dots,d)$, so 
\be\label{con_eq1}
g(\tfrac12\e_1+\tfrac12\e_2)\ge
\tfrac12 g(\e_1)+\tfrac12 g(\e_2)=
g(\e_2).
\ee
If we knew that $g$ was strictly concave, then it would follow immediately that inequality \eqref{con_eq1} was strict. However, it is an open problem to determine when $g$ is strictly concave.
\begin{lem}\label{con_lem1}
Inequality \eqref{con_eq1} is strict if the vertex weights are random.
\end{lem} 
\begin{proof}
As there are two directed paths from $\0$ to $\e_1+\e_2$, but only one directed path from $0$ to $2\e_2$, 
\[
\EE[T(\0,\e_1+\e_2)] > \EE [T(\0,2\e_2)] .
\]
Let $N\in\ZZ_+$. 
One can construct a directed path from $\0$ to $2N(\e_1+\e_2)$ by joining $N$ paths with displacement $\e_1+\e_2$, so
\begin{align*}
\EE [T(\0,N(\e_1+\e_2))] \ge N \EE [T(\0,\e_1+\e_2)] > N\EE [T(\0,2\e_2)] = \EE [T(\0,2N\e_2)].
\end{align*}
Divide by $2N$ and let $N\to\oo$.
\end{proof}
Again by symmetry, if $d>2$ and $x_1=x_2=0$,
\[
g(\tfrac12\e_1+\tfrac12\e_2+\x) \ge g(\e_2+\x).
\]
Even with random vertex weights, this inequality is not necessarily strict---consider Bernoulli vertex weights with density $p$ \cite{MartinSurvey}. When $p$ is sufficiently close to $1$, the process is supercritical in the sense of ordinary directed site percolation. In that case, $g(\x)$ reaches a plateau of $1$ as a function on the simplex $\{\x\in\RR_+^d:|\x|=1\}$.

The behavior in the Bernoulli case seems to be the exception rather than the rule. The Bernoulli distribution is `unusual' in that it places a positive amount of mass on a maximum value. For contrast, consider the geometric distribution, which is unbounded, and the uniform $[0,1]$ distribution, which is bounded but places zero mass at 1. 
\begin{lem}\label{not_flat}
Suppose the vertex weight distribution does not place mass on a maximum value, and that \eqref{shape proposition condition} is satisfied. There exists $c>0$ such that
\[
g(\tfrac12\e_1+\tfrac12\e_2+\x) \ge g(\e_2+\x)+c,\qq \forall\,\x\in\{0\}^2\times [0,1]^{d-2}.
\]
\end{lem}
\begin{proof}
Let $\x\in\{0\}^2\times[0,1]^{d-2}$. 
With $N$ a positive integer, let $\p=N\e_2+\lfloor N\x\rfloor$
and let $L$ denote the line segment
\[
L=\{\p+\al(-\e_2+\e_1):\al\in[0,N]\}.
\]
We will use $\om$ to construct a second set of vertex weights: $\phi$. Let $T_\phi$ denote the last passage times under $\phi$.
In the process of constructing $\phi$, we will also construct a function $\x:\{0\}\times\zzdm_+\to\zzd_+$ such that 
$\x(\p)$ lies on the line $L$. 
The construction is designed to allow the comparison of $T(\0,\p)$ with $T_\phi(0,\x(\p))$. 
By Proposition \ref{shape proposition} and the concavity of $g$, it is enough to show that for some $c>0$, 
\[
\lim_{N\to\oo} \PP\Big(T_\phi(\0,\x(\p)) \ge T(\0,\p) + c N\Big)=1.
\]
We will now be more precise. 
Let $\v$ represent an element of $\{0\}\times\zzdm_+$. The construction will ensure that, 
\be\label{x_of_v}
\x(\v)\in\{\v+\al (-\e_2+\e_1):\al\in[0,v_2]\} \qtext{and} T_\phi(\0,\x(\v))
\ge T(\0,\v).
\ee 
We will construct $\x$ and $\phi$ inductively. Notice that the last-passage times $(T(\0,\v):\v\in\zzd_+)$ satisfy an inductive relationship,
\[
T(\0,\v)= \om(\v)+\max_{j=1,2,\dots,d}  T(\0,\v-\e_j).
\]
Let $I_k=\{\v\in \{0\}\times\zzdm_+:|\v|=k\}$. To begin the process, let $\x(\0)=\0$ and $\phi(\0)=\om(\0)$. 
Now assume inductively that for $\v\in I_{k-1}$, $\x(\v)$ and $\phi(\v)$ have been defined in accordance with \eqref{x_of_v}.
We will carry out the inductive step in three stages.

First, consider separately all $\v\in I_k$. Choose $j=j_\v\in\{2,3,\dots,d\}$ to maximize
$T(\0,\v-\e_j)$.  Let
\begin{align*}
\hat{\x}(\v)&=\x(\v-\e_j)+\e_j, \text{ and}\\ \phi(\hat{\x}(\v))&=\om(\v).
\end{align*}
Hence $T_\phi(\0,\hat{\x}(\v))\ge T(\0,\v)$.  
Second, for all $\v\in\zzd$ with $|\v|=k$, if $\phi(\v)$ is undefined after the first stage, take $\phi(\v)$ to be an auxiliary vertex weight random variable, independent of everything else.
Third, to finish off the inductive step, consider again all $\v\in I_k$. If $j_\v>2$, set $\x(\v)=\hat{\x}(\v)$. If $j_\v=2$, $\phi(\hat{\x}(\v)-\e_2+\e_1)$ is one of the auxiliary random variables; set
\[
 \x(\v)=\begin{cases} 
\hat{\x}(\v) & \text{if }
\phi(\hat{\x}(\v))>\phi(\hat{\x}(\v)-\e_2+\e_1),\\ \hat{\x}(\v)-\e_2+\e_1 & \text{if }
\phi(\hat{\x}(\v))<\phi(\hat{\x}(\v)-\e_2+\e_1).\\
\end{cases}
\]
Now $T_\phi(\0,\x(\v))\ge T_\phi(\0,\hat{\x}(\v))\ge T(\0,\v)$. By induction in $k$, \eqref{x_of_v} holds.

The $\om$-maximal path from $\0$ to $\p$ contains $N$ steps
of the form $\v-\e_2\to\v$; let $\v_i-\e_2\to\v_i$, $i=1,\dots,N$, denote the $N$ steps. Let
\[
A_i=\phi(\hat{\x}(\v_i)),\qq B_i=\phi(\hat{\x}(\v_i)-\e_2+\e_1).
\]
Note that if $\#\{i:B_i \ge A_i + \eps\}\ge \d N$, then $T_\phi(\0,\x(\p))\ge T(\0,\p)+\eps\d N$. 
Whilst the $B_i$ are typical vertex weights, independent of $\om$, the $A_i$ are vertex weights on the $\om$-maximal path from $\0$ to $\p$. 

The number of directed paths from $0$ to $\p$ is less than $d^{\,|\p|}\le d^{\,dN}$. 
Let $F$ denote the vertex weight cumulative distribution function, and let 
\[
a=1/(16d^{\,2d}),\qquad t_a=F^{-1}(1-a),\qquad t_b\in\{t>t_a:F(t)<1\}.
\]
The probability that more than half of the vertex weights associated with steps in the direction $\e_2$ along any directed path from $\0$ to $\p$ are greater than $t_a$ is less than $d^{\,dN} (2^N a^{N/2}) \le 2^{-N}$. 
Taking $0<\d<(1/2)(1-F(t_b))$,
\[
\lim_{N\to\oo}\PP\big(\#\{i:A_i \le t_a < t_b\le B_i\}\ge \d N\big)=1.\qedhere
\]
\end{proof}

\section{Proof of Theorem \ref{variance_thm}}\label{variance_section}
The proof is based on the corresponding result for first-passage
percolation with Bernoulli-type edge weights \cite{BKS}; their proof exploits the fact that first-passage times form a metric space, so $\nt{S(0,\x+\e_1)-S(\0,\x)}$ is bounded as $|\x|\to\oo$. The main difficulty in adapting the proof for the directed case is that $T$ is not a metric.
We will show that:
\begin{lem}\label{shift_lemma}
$\nt{T(-\e_1,N\u)-T(\0,N\u)}=\O(N^{1/4}\log N)$.
\end{lem}
The proof of Lemma \ref{shift_lemma} uses Lemmas \ref{conc_d} and \ref{not_flat}.
Before we give the proof, we will show how it is used to deduce Theorem \ref{variance_thm}.
\begin{proof}[Proof of Theorem \ref{variance_thm}]
Let $m=\lfloor N^{1/8}\rfloor$, let $S= \{1,\dots,d m^2\}$,
and let $(\PP,\Om)$ be defined as in Lemma \ref{conc_lem}. By \cite[Lemma 3]{BKS} there exists a constant $c$, independent of $m$, and a random variable $\Z:\Om\to\{1,\dots,m\}^d$, such that
\begin{romlist}
\item $\Z$ is independent of $\{\om(\v):\v\in\zzd\}$,
\item if $\om(s)=\om'(s)$ for all but one $s\in S$, $|\Z(\om)-\Z(\om')|= 1$, and
\item for all $\z$, $\PP(\Z(\om)=\z) \le (c/m)^d$.
\end{romlist}
Let $f(\om)=T(\Z,N\u+\Z)$; by translation invariance, $\var[f]=\var[T(\0,N\u)]$. The effect of randomizing the start and end points is to spread out the influence of any given vertex weight. Let $\v\in\zzd$. The range of $\v-\Z$ is
\[
\text{range}[\v-\Z]=\v-\{1,\dots,m\}^d.
\]
If $\eta$ is a directed path from $\0$ to $N\u$, the intersection of $\eta$ and $\text{range}[\v-\Z]$ has size at most $md$, and so
\[
\PP(\v\in\eta+\Z) \le md (c/m)^d = 1/\,\Om(m).
\]
Let $\gamma$ denote the $\om$-maximal path corresponding to $f$ from $\Z$ to $N\u+\Z$. By the independence of $\Z$ and $\gamma-\Z$,
\[
\ivf=\PP(\v\in \gamma) = \sum_\eta \PP(\gamma=\eta+\Z) \PP (\v \in \eta+\Z)=1/\,\Om(m).
\]
Hence $\max_\v \ivf = 1/\Om(m)$. By Lemma \ref{conc_lem}, as $|N\u|=\O(N)$, 
\begin{align}\label{almost done}
\var[f]\log\frac{\var[f]}{\O(N^{7/8}) + I_S(f)} \le \O(N) + 2I_S(f).
\end{align}
For each $s\in S$, changing the value of $\om(s)$ causes $\Z$ to move a unit distance. By translation invariance and the triangle inequality,
\begin{align*}
2\nt{\De_s f} =&\nt{T(-\e_1,N\u-\e_1)-T(\0,N\u)}\\
\le &\nt{T(-\e_1,N\u)-T(0,N\u)}+\nt{T(-\e_1,N\u-\e_1)-T(-\e_1,N\u)}.
\end{align*}
Hence, by Lemma \ref{shift_lemma}, we can bound
\begin{align*}
I_S(f)=\sum_{s\in S} \nt{\De_s f}^2
=dm^2\,\O(N^{1/4}\log N)^2=\O(N^{7/8}).
\end{align*}
The result follows by \eqref{lambert} and \eqref{almost done}.
\end{proof}
It remains to prove Lemma \ref{shift_lemma}. 
Note that if two random variables have Gaussian tails, so does their sum:
\begin{remk}\label{convolution}
Let $X_1,X_2$ denote random variables such that
\[
\PP[X_i\ge t] =\exp(-\Om(t^2)), \qq t\ge 0;\quad i=1,2.
\]
Then,
\[
\PP[X_1+X_2\ge t] =\exp(-\Om(t^2)),\qq t\ge 0.
\]
\end{remk}

\begin{proof}[Proof of Lemma \ref{shift_lemma}]
Let $\gamma$ denote the $\om$-maximal path from $-\e_1$ to $N\u$. The key to proving this lemma is showing that almost all of the path $\gamma$ is accessible from $\0$. 
Let $\a$ denote a point on the hyperplane $x_1=0$:
\begin{align*}
\0\le\a\le N\e_2+\dots+N\e_d.
\end{align*}
Let $h(\x)=T(-\e_1,\x)-T(\0,\x)$. 
If $\a\in\gamma$ then $h(N\u)\le h(\a)$ and so
\begin{align}\label{claimb}
&\PP\left(h(N\u)\ge 2 tN^{1/4}\right)\nonumber\\
&\le \sum_\a \PP\left(\a\in\gamma \text{ and } h(\a) \ge 2 tN^{1/4}\right)\nonumber\\
& \le 
\sum_{\a:|\a|^2<t^{4/3}N}\PP\left(h(\a) \ge 2 tN^{1/4}\right)
+ 
\sum_{\a:|\a|^2\ge t^{4/3}N}\PP(\a\in\gamma).
\end{align}
Let $c$ represent a positive constant and take $t\ge (c\log N)^{3/4}$.

The random variable $h(\a)$ is the difference of two last-passage times. By Lemma \ref{conc_d} and Remark \ref{convolution},
\be\label{claimc}
\PP\left(h(\a) -\EE [ h(\a) ] \ge tN^{1/4}\right)=\exp\left(-\Om\left(\frac{t^2N^{1/2}}{|\a|\vee 1}\right)\right).
\ee
To find an upperbound for $\EE [ h(\a) ]$, note that for some $\a'$ such that $\0\le \a'\le \a$,
\[
T(-\e_1,\a)=T(-\e_1,-\e_1+\a')  + T(\a',\a) + \om(\a').
\]
For each fixed $\a'$, the expected value of the right-hand side is at most $\EE[T(0,\a)+\om(\0)]$.
The number of values $\a'$ can take is less than $N^d$.
If $|\a|^2< t^{4/3}N$ and $c$ is sufficiently large then Lemma \ref{conc_d} yields
\be\label{claimd}
\EE [ h(\a) ] = \O\left(\sqrt{|\a|\log (N^d)}\right) \le tN^{1/4}.
\ee
Assuming again that $c$ is sufficiently large, we claim that
\be\label{claim}
|\a|^2\ge t^{4/3} N \implies \PP(\a\in\gamma)=\exp\left(-\Om\left({|\a|^2}/{N}\right)\right).
\ee
We can assume without loss of generality that $\a=a_2\e_2+\dots+a_d\e_d$ with $a_2\ge |\a|/(d-1)$.  Let $\b$ be defined by
\begin{align*}
\b  = \a+\left\lfloor\frac{a_2}2\right\rfloor (\e_1-\e_2),
\end{align*}
see Figure \ref{fig:thm2}. Note that $0\le \a,\b\le N\u$ and $|\a|=|\b|$.
\begin{figure}[tb]
\begin{center}
\begin{picture}(0,0)%
\includegraphics{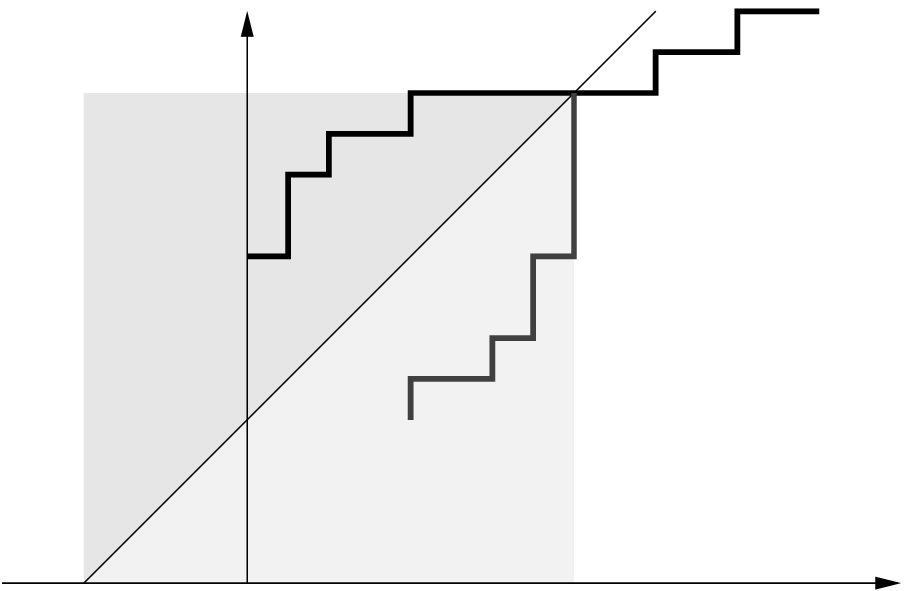}%
\end{picture}%
\setlength{\unitlength}{3439sp}%
\begingroup\makeatletter\ifx\SetFigFont\undefined%
\gdef\SetFigFont#1#2#3#4#5{%
  \reset@font\fontsize{#1}{#2pt}%
  \fontfamily{#3}\fontseries{#4}\fontshape{#5}%
  \selectfont}%
\fi\endgroup%
\begin{picture}(4974,3240)(-1361,-15403)
\put(361,-13831){\makebox(0,0)[lb]{\smash{{\SetFigFont{10}{12.0}{\familydefault}{\mddefault}{\updefault}{\color[rgb]{0,0,0}$L$}%
}}}}
\put(1666,-12571){\makebox(0,0)[lb]{\smash{{\SetFigFont{10}{12.0}{\familydefault}{\mddefault}{\updefault}{\color[rgb]{0,0,0}$\q$}%
}}}}
\put(-134,-13606){\makebox(0,0)[lb]{\smash{{\SetFigFont{10}{12.0}{\familydefault}{\mddefault}{\updefault}{\color[rgb]{0,0,0}$\a$}%
}}}}
\put(811,-14641){\makebox(0,0)[lb]{\smash{{\SetFigFont{10}{12.0}{\familydefault}{\mddefault}{\updefault}{\color[rgb]{0,0,0}$\b$}%
}}}}
\put(3196,-12346){\makebox(0,0)[lb]{\smash{{\SetFigFont{10}{12.0}{\familydefault}{\mddefault}{\updefault}{\color[rgb]{0,0,0}$N\u$}%
}}}}
\put( 46,-15271){\makebox(0,0)[lb]{\smash{{\SetFigFont{10}{12.0}{\familydefault}{\mddefault}{\updefault}{\color[rgb]{0,0,0}$\0$}%
}}}}
\put(-314,-12346){\makebox(0,0)[lb]{\smash{{\SetFigFont{10}{12.0}{\familydefault}{\mddefault}{\updefault}{\color[rgb]{0,0,0}$x_2$}%
}}}}
\put(3376,-15226){\makebox(0,0)[lb]{\smash{{\SetFigFont{10}{12.0}{\familydefault}{\mddefault}{\updefault}{\color[rgb]{0,0,0}$x_1$}%
}}}}
\end{picture}%
\end{center}
\caption{The reflection $\q\b$ of the path corresponding to $T(\a,N\u)$ when $d=2$. When $a_2$ is even, $\b$ lies exactly on the line $x_1=x_2$.\label{fig:thm2}}
\end{figure}
Let
\begin{align*}
D_1&=T(\a,N\u)-T(\b,N\u),\\
D_2&=T(-\e_1,\a)- T(-\e_1,\b).
\end{align*} 
We have chosen $D_1$ and $D_2$ so that $\PP(\a\in\gamma)\le \PP(D_1+D_2>0)$. 
By Remark \ref{convolution} and Lemma \ref{conc_d},
\[
\PP(D_1 + D_2 \ge \EE[D_1+D_2]+t\sqrt{N}) =\exp(-\Om(t^2)),\qq t\ge 0.
\]
To check \eqref{claim} it is sufficient to show that $\EE[D_1+D_2]=-\Om(|\a|)$ for $\a$ such that  $|\a|\ge t^{2/3}N^{1/2} \ge  \sqrt{c N\log N}$. If $c$ is sufficiently large, this follows from the bounds
\be\label{work}
\EE [D_1] = \O(\sqrt {N \log N})\qtext{and} \EE [D_2] =-\Om(|\a|).
\ee
Note that \eqref{work} only bounds $\EE[D_1]$ above; $\EE[D_1]$ may well be negative.

Let $L$ denote the hyperplane $x_2=x_1+\lceil a_2/2\rceil$; $\b$ is the reflection in $L$ of $\a$. All paths from $\a$ to $N\u$ must pass through $L$. Choose $\q\in L$ such that the $\om$-maximal path from $\a$ to $N\u$ passes through $\q$ with probability at least $N^{1-d}$.
Let $R$ denote the set of points that are less than $\q$ in the partial order, and that are on the opposite side of $L$ to $\0$. Let $R'$ be the reflection of $R$ in $L$.  
Consider the bijection $\zzd\lra\zzd$ obtained by reflecting $R\lra R'$ in $L$. If the $\om$-maximal path from $\a$ to $N\u$ passes through $\q$, then the bijection produces a configuration $\om'$ such that \[T_{\om'}(\b,N\u)\ge T(\a,N\u).\] 
Let $F_\x$ denote the cumulative distribution function of $T(\x,N\u)$. 
The bijection implies that \[F^{-1}_{\b}(1-N^{1-d})\ge F^{-1}_\a(N^{1-d}),\] and so by concentration $\EE [D_1] =\O(\sqrt{N\log N})$.

The upper bound on $\EE [D_2]$ follows from Lemma \ref{not_flat} and Proposition \ref{shape proposition}.

By \eqref{claimb}--\eqref{claim}, when $t\ge (c\log N)^{3/4}$
\begin{align*}
\PP\left(h(N\u)\ge 2 tN^{1/4}\right)\le t^{2d/3}N^{d/2}\exp(-\Om(t^{4/3})).
\end{align*}
Hence $\nt{h(N\u)}=\O(N^{1/4}\log N)$.
\end{proof}

\section{Directed polymers}\label{sec dp}
We have considered directed last passage percolation on $(\zzd,\vec{E})$ \eqref{E def}. Piza \cite{Piza} considered the directed polymer model on the same directed graph. Let $S_N=\{\x\in\zzd_+:|\x|=N\}$ denote the set of points withing reach of $\0$ in $(\zzd,\vec{E})$ by directed paths of length $N$. The weight of the polymer in the ground state of the directed polymer model is 
\[
T(\0,S_N):=\sup_{\x\in S_N} T(\0,\x).
\]
The directed polymer model is more often studied on an alternative directed graph.
Think of $\zzd$ as containing $d-1$ space dimensions, with undirected nearest-neighbor edges, and one time dimension. 
This corresponds to a set of directed edges,
\[
\hat{E}=\hat{E}(\zzd)=\{(\u,i)\to(\v,i+1):\u,\v\in\zzdm,|\u-\v|=1\}.
\]
In the context of $(\zzd,\hat{E})$, we will write $\hat{T}(\x,\y)$ for the last passage time from $\x$ to $\y$, $\hat{g}(\x)$ for the asymptotic last passage time in direction $\x$, and $\hat{S}_N$ for the set of points at graph distance $N$ from $\0$.
\begin{thm}\label{variance_thm2}
Let $d\ge 2$, and let $\om=(\om(\x):\x\in\zzd)$ be
a collection of independent, standard normal random variables; 
\[
\var[T(\0,{S}_N)]=\O(N/\log N)\qtext{and} \var[\hat{T}(\0,\hat{S}_N)]=\O(N/\log N).
\]
\end{thm}
\begin{proof}
The proof of Theorem \ref{variance_thm} can be applied almost directly. We will first show that $\var[T(\0,S_N)]=\O(N/\log N)$.
Note that we can painlessly replace $\x$ with $S_N$ in Lemma \ref{conc_d},
\begin{align*}
\PP\Big(
\big|T(\0,{S}_N)-\EE T(\0,{S}_N)\big|
\ge t\sqrt{N}\Big)
=\exp(-\Om(t^2)).  
\end{align*}
If we replace $N\u$ with $S_N$ in Section \ref{variance_section}, the only modification needed relates to the upper bound on $\EE [D_1]$ in the proof of Lemma \ref{shift_lemma}. Replacing $N\u$ with $S_N$ gives $D_1=T(\a,S_N)-T(\b,S_N)$. By translation invariance, $\EE [D_1]=0$. This satisfies \eqref{work} and so $T(\0,S_N)$ has variance sublinear in $N$.

We will now prove the corresponding result for $(\zzd,\hat{E})$. Note that while the asymptotic traversal time function $g$ for $(\zzd,\vec{E})$ is defined on $\RR_+^d$, $\hat{g}$ is defined on the `space-time' cone $\sC=\{\x:x_d\ge |x_1|+\dots+|x_{d-1}|\}$. 

Recall that $(\ZZ^2,\vec{E})$ and $(\ZZ^2,\hat{E})$ are equivalent, up to a rotation by 45$^\circ$.
More generally, if $d\ge 2$, identify the directions $\e_1,\e_2,\dots,\e_{d-1},\e_d$ of $(\zzd,\vec{E})$ with the directions $\e_1+\e_d,\e_2+\e_d,\dots,\e_{d-1}+\e_d,-\e_1+\e_d$ in $(\zzd,\hat{E})$; this embeds $(\zzd,\vec{E})$ into $(\zzd,\hat{E})$ as a subgraph.
\begin{lem}\label{bar_not_flat}
Suppose the conditions of Lemma \ref{not_flat} are satisfied. 
Let $\hat{\a}$ denote a point on the surface of $\sC$.
There is a point $\hat{\b}$ in $\sC$ such that $\hat{a}_d=\hat{b}_d$ and
\begin{align*}
\EE[\hat{T}(\0,\hat{\b})]\ge\EE[\hat{T}(\0,\hat{\a})]+\Om(|\hat{\a}|).
\end{align*}
\end{lem}
\begin{proof}
Without loss of generality, assume $a_1\ge a_2,\dots,a_{d-1}\ge 0$. 
Let
\begin{align*}
\a=\hat{\a}-a_d\e_d, \qq \b=\a+\lfloor a_1/2\rfloor(-\e_1+\e_d), \qq \hat{\b}=\hat{\a}-2\lfloor a_1/2\rfloor\e_1. 
\end{align*}
The embedding produces a bijection between the set of $\vec{E}$-paths from $\0$ to $\a$ and the set of $\hat{E}$-paths from $\0$ to $\hat{\a}$:
\[
\EE [T(\0,\a)] =\EE [\hat{T}(\0,\hat{\a})].
\]
The embedding produces an injective function from the set of $\vec{E}$-paths from $\0$ to $\b$ to the set of $\hat{E}$-paths from $\0$ to $\hat{\b}$:
\[
\EE [T(\0,\b)] \le \EE [\hat{T}(\0,\hat{\b})].
\]
Use Lemma \ref{not_flat} with Proposition \ref{shape proposition} to compare $\EE [T(\0,\a)]$ and $\EE [T(\0,\b)]$.
\end{proof}
Take $m$, $S$, $\EE$, $\Om$ and $\Z$ from the proof of Theorem \ref{variance_thm}. 
Let $\hat{\Z}$ denote the random vector obtained by applying to $\Z$ the embedding from $(\zzd,\vec{E})$ into $(\zzd,\hat{E})$; let $\hat{f}=\hat{T}(\hat{\Z},\hat{\Z}+\hat{S}_N)$.
Following the proof of Theorem \ref{variance_thm}, in place of Lemma \ref{shift_lemma} we must show that, say, 
\[
\nt{\hat{T}(\e_1-\e_d,\hat{S}_N)-\hat{T}(\0,\hat{S}_N)}=\O(N^{1/4}\log N).
\]
To adapt the proof of Lemma \ref{shift_lemma}, we must show that almost all of the maximal path $\gamma$ from $\e_1-\e_d$ to $\hat{S}_N$ lies inside the space-time cone generated by the origin, $\sC$.
Let $\hat{\a}$ represent a point on the surface of $\sC$, representing where $\gamma$ enters the cone. Choose $\hat{\b}$ using Lemma \ref{bar_not_flat}; setting
\begin{align*}
D_1 &= \hat{T}(\hat{\a},\hat{S}_N)-\hat{T}(\hat{\b},\hat{S}_N),\qq D_2 = \hat{T}(\e_1-\e_d,\hat{\a})-\hat{T}(\e_1-\e_d,\hat{\b}),
\end{align*}
we get $\EE[D_1]=0$ and $\EE[D_2] =-\Om(|\hat{\a}|)$.
The remainder of the proof holds mutatis mutandis.
\end{proof}

\section{Other vertex weight distributions}\label{other}
So far, we have  restricted our attention to Gaussian vertex weights.
The proof of Theorem \ref{variance_thm} can be modified to accommodate other vertex weight distributions. 
For simplicity, we will just look at two examples: the uniform $[0,1]$ distribution and the gamma distribution. 

Let $S$ denote a countable set, and let $\EE$ denote a product measure that supports the vertex weights $(\om(\v):\v\in\zzd)$, and auxiliary Bernoulli($1/2)$ random variables $(\om(s):s\in S)$. As in the Gaussian case, define the influence of vertex $\v\in\zzd$ to be
\[
\ivf=\PP\big(\pfo\not=0\big).
\]
\subsection{Uniform $[0,1]$ vertex weights}
By \cite[Corollary 2.4]{BR}, for $f\in H_1^2$,
\begin{align*}
\var[f]\log\frac{\var[f]}{\sum_\v \no{\De_\v f}^2 +\sum_s \no{\De_s f}^2} \le \frac{2}{\pi^2} \sum_\v \ivf +2\sum_s \nt{\De_s f}^2.
\end{align*}
This effectively replaces \cite[Corollary 2.2]{BR} in the proof of Lemma \ref{conc_lem}. 
Inequality \eqref{cg2} is now much simpler; it is easily seen to hold with $c_G$ replaced by $c_U=1/2$. Theorem \ref{variance_thm} therefore also applies with uniform $[0,1]$ vertex weights.

\subsection{Gamma $\Gamma(\al,\beta)$ vertex weights}
\def\rvf{J_\v(f)}
Define `weighted' influences
\[
\rvf= \nt{\pfo \sqrt{1+\om(\v)}}^2,\qq \v\in\zzd.
\]
Let $f(\om)=T(\Z,\x+\Z)\wedge B\vee A$ with $0\le A < B \le \oo$. 
By Corollary 2.3 of \cite{BR}, there is a constant $C_{\al,\beta}$ such that,
\begin{align*}
\var[f]\log&\frac{\var[f]}{
\sum_\v \no{\De_\v f}^2 +
\sum_{s\in S} \no{\De_s f}^2
}
\le 
C_{\al,\beta}\sum_\v \rvf+
2\sum_{s\in S} \nt{\De_s f}^2.
\end{align*}
Let $X=\om(\0)$ so that $X$ represent a $\Gamma(\al,\beta$) random variable. We can check that $c_{\Gamma(\al,\beta)}<\oo$, where
\[
c_{\Gamma(\al,\beta)}:= \sup_{0\le a<b\le \oo} \frac{\EE|X\wedge b \vee a - \EE(X\wedge b \vee a)|}{\EE[(1+X) 1_{\{a<X<b\}}]}.
\]
In place of Lemma \ref{conc_lem}, we have 
\begin{align*}
\var[f]\log\frac{\var[f]}{c_{\Gamma(\al,\beta)}^2\sum_\v \rvf^2 +I_S(f)} &\le C_{\al,\beta}\sum_\v \rvf +2 I_S(f).
\end{align*}
Note that
\begin{align*}
\sum_\v \rvf &\le 
\begin{cases} 
u(|\x|+B) &\text{if } B<\oo\\ 
|\x|+|\x|g(\e_1+\dots+\e_d) & \text{if } B=\oo,\\
\end{cases}
\end{align*}
and, as the $\Gamma(\al,\beta)$ distribution has an exponential tail,
\[
\rvf/\ivf =\O\Big( \log 1/\ivf\Big), \qq \v\in\zzd.
\]
Adapting the proof of Lemma \ref{conc_d} gives a slightly weaker concentration inequality:
for $\x\in\zzd_+$, 
\[
\PP\Big(
|{T}(\0,\x)-\EE {T}(\0,\x)|
\ge t\Big)
=
\begin{cases}
\exp(-\Om(t^2/|\x|)), & t \le |\x|,\\
\exp(-\Om(t)), & t \ge |\x|.\\
\end{cases}  
\]
Nonetheless, this is sufficient for following the argument in Section \ref{variance_section}.
Take $f(\om)=T(\Z,N\u+\Z)$ from the proof of Theorem \ref{variance_thm}. Following the proof gives $\max_\v\ivf=1/\,\Om(m)$, and hence that
\[
\sum_\v \rvf^2 \Big/\sum_\v \rvf  \le \max_\v\rvf=1/\,\Om(m/\log m).
\]
Once again, $\var[f]=\O(N/\log N)$.

\section{Acknowledgement}
Thanks to Lung-Chi Chen, Geoffrey Grimmett and the referee for their helpful comments about the problem.

\def\cprime{$'$}

\end{document}